\documentclass[a4paper,11pt]{article}
\usepackage[utf8]{inputenc}

\usepackage{amsthm}
\usepackage{amsmath}
\usepackage{amssymb}
\usepackage{lineno}

\usepackage{hyperref}
\usepackage{enumerate}
\usepackage{graphicx}
\usepackage{nicefrac}
\usepackage{complexity}

\usepackage[usenames,dvipsnames]{xcolor}

\usepackage[font=small,labelfont=bf]{caption}
\usepackage[font=small,labelfont=normalfont,labelformat=simple]{subcaption}

\DeclareMathOperator{\conv}{conv}
\DeclareMathOperator{\sgn}{sgn}

\newtheorem{theorem}{Theorem}
\newtheorem{lemma}{Lemma}
\newtheorem{corollary}{Corollary}

\graphicspath{{figs/}}

\title{Two Disjoint 5-Holes in Point Sets\footnote{%
An extended abstract of this work 
was presented at the 35th European Workshop on Computational Geometry (EuroCG’19) \cite{disjoint_holes_eurocg_version}.
A short version of this work (8 pages) 
 is to appear in the Proc.\ of
the European Conference on Combinatorics, Graph Theory and Applications (EUROCOMB'19) \cite{disjoint_holes_eurocomb_version}.
}}
\author{
  Manfred Scheucher
}
\date{\normalsize Institut f\"ur Mathematik, \\
    Technische Universit\"at Berlin, \\
    Berlin, Germany\\
    \vspace{0.2cm}
    \texttt{\{scheucher\}@math.tu-berlin.de}}

\begin{document}

\maketitle

\begin{abstract}
Given a set of points $S \subseteq \mathbb{R}^2$, 
a subset \mbox{$X \subseteq S$} with $|X|=k$ is called 
\mbox{\emph{$k$-gon}} if all points of $X$ lie on the boundary of the convex hull of~$X$,
and \emph{$k$-hole} if, in addition, no point of $S \setminus X$ lies in the convex hull of~$X$.
We use computer assistance to show that every set of 17 points in general position 
admits two \emph{disjoint} 5-holes, that is, holes with disjoint respective convex hulls.
This answers a question of Hosono and Urabe (2001).
We also provide new bounds for three and more pairwise disjoint holes.

In a recent article, Hosono and Urabe (2018) 
present new results on interior-disjoint holes -- a variant,
which also has been investigated in the last two decades. 
Using our program, we show
that every set of 15 points contains two interior-disjoint 5-holes.

Moreover, our program can be used to verify 
that every set of 17 points contains a 6-gon
within significantly smaller computation time
than the original program by Szekeres and Peters (2006).
Another independent verification of this result 
was done by Mari\'c (2019).
\end{abstract}


\section{Introduction}

A set $S$ of points in the Euclidean plane is \emph{in general position}
if no three points lie on a common line.
Throughout this paper all point sets are considered to be finite and in general position.
A subset $X \subseteq S$ of size $|X| = k$ is a \mbox{\emph{$k$-gon}} 
if all points of $X$ lie on the boundary of the convex hull of~$X$, denoted by $\conv(X)$.
A classical result of Erd\H{o}s and Szekeres from the 1930s asserts that, for fixed $k \in \mathbb{N}$, 
every set of $\binom{2k-4}{k-2}+1$ points contains a $k$-gon~\cite{ErdosSzekeres1935} (cf.~\cite{Matousek2002}).
They also constructed point sets of size $2^{k-2}$ with no $k$-gon.
There were several small improvements on the upper bound by various researchers in the last decades, 
each of order $4^{k-o(k)}$,
until Suk~\cite{Suk2017} significantly improved the upper bound to $2^{k+o(k)}$.
However, the precise minimum number $g(k)$ 
of points needed to guarantee the existence of a $k$-gon is still 
unknown for $k\ge7$ (cf.~\cite{SzekeresPeters2006})\footnote{%
Erd\H{o}s offered \$500 for a proof of Szekeres' conjecture that $g(k) = 2^{k-2}+1$.}.

In the 1970s, Erd\H{o}s~\cite{Erdos1978} asked 
whether every sufficiently large point set
contains a \emph{$k$-hole}, that is, 
a $k$-gon with no other points of $S$ lying inside its convex hull.
Harborth~\cite{Harborth78} showed that every set of 10 points contains a 5-hole
and Horton~\cite{Horton1983} introduced a construction of arbitrarily large point sets without 7-holes.
The question, whether \mbox{6-holes} exist in sufficiently large point sets, remained open until 2007, 
when Nicolas~\cite{Nicolas2007} and Gerken~\cite{Gerken2008} 
independently showed that point sets with large $k$-gons also contain $6$-holes\footnote{%
For a reasonably short proof for the existence of 6-holes we refer to~\cite{Valtr2009}.}.
In particular, Gerken proved that every point set
that contains a 9-gon also contains a \mbox{6-hole}.
The currently best upper bound on the number of points is by Koshelev~\cite{Koshelev2009a}\footnote{%
Koshelev's publication covers more than 50 pages (written in Russian)},
who showed that every set of 463 points contains a \mbox{6-hole}.
However, the largest set without \mbox{6-holes} currently known has 29 points 
and was found using a simulated annealing-based approach by Overmars~\cite{Overmars2002}.

In 2001, Hosono and Urabe~\cite{HOSONO200197} 
and B{\'a}r{\'a}ny and K{\'a}rolyi \cite{BaranyiKarolyi2001}
started the investigation of disjoint holes,
where two holes $X_1,X_2$ of a given point set $S$ 
are said to be \emph{disjoint}
if their respective convex hulls are disjoint 
(that is, $\conv(X_1) \cap \conv(X_2) = \emptyset$; see Figure~\ref{fig:disjoint_holes}).
This led to the following question: 
What is the smallest number $h(k_1,\ldots,k_l)$ such that
every set of $h(k_1,\ldots,k_l)$ points determines
a $k_i$-hole for every \mbox{$i=1,\ldots ,l$},
such that the holes are pairwise disjoint~\cite{Hosono2008}?
As there are arbitrarily large point sets without 7-holes,
only parameters $k_i < 7$ are of interest.
Moreover, 
since the gap between the upper bound and the lower bound for $h(6)$ is still huge, 
mostly values with parameters $k_1,\ldots,k_l \le 5$ were investigated.
Also note that, if all $k_i$ are at most 3, 
then the value $h(k_1,\ldots,k_l)=  k_1 + \ldots + k_l$ is straight-forward
because every set of $k_1+\ldots+k_l$ points can be 
cut into blocks of $k_1,\ldots,k_l$ points (from left to right), 
which clearly determine the desired holes.

\begin{figure}
 
	\begin{subfigure}[t]{.2\textwidth}
		\centering
		\includegraphics[page=1]{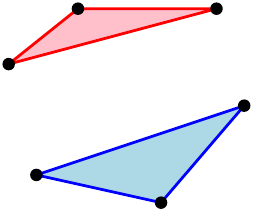}
		\caption{}
		\label{fig:disjoint_holes1}  
	\end{subfigure}
	\hfill
	\begin{subfigure}[t]{.2\textwidth}
		\centering
		\includegraphics[page=2]{disjoint_holes}
		\caption{}
		\label{fig:disjoint_holes2}  
	\end{subfigure}
	\hfill
	\begin{subfigure}[t]{.2\textwidth}
		\centering
		\includegraphics[page=3]{disjoint_holes}
		\caption{}
		\label{fig:disjoint_holes3}  
	\end{subfigure}
	\hfill
	\begin{subfigure}[t]{.2\textwidth}
		\centering
		\includegraphics[page=4]{disjoint_holes}
		\caption{}
		\label{fig:disjoint_holes4}  
	\end{subfigure}
	
	\caption{
	The two holes depicted in \subref{fig:disjoint_holes1} are disjoint 
	while the ones in \subref{fig:disjoint_holes2}--\subref{fig:disjoint_holes4} are not disjoint.
	Moreover, the holes depicted in \subref{fig:disjoint_holes1}--\subref{fig:disjoint_holes3} 
	are interior-disjoint 
	while the ones in \subref{fig:disjoint_holes4} are not interior-disjoint.
	(The notion of interior-disjoint holes will be introduced and discussed in Section~\ref{sec:final_remarks}).
	}
	\label{fig:disjoint_holes}
\end{figure}

\medskip

In Sections~\ref{sec:results_two} and~\ref{sec:results_three},
we summarize the current state of the art 
for two- and three-parametetric values
and we present some new results 
that were obtained using computer-assistance.
Our main contribution is 
that every set of 17 points contains two disjoint 5-holes (Theorem~\ref{thm:h55}).
Moreover, we describe some direct consequenses for multi-parametric values 
in Section~\ref{sec:results_many}.
The basic idea behind our computer-assisted proofs 
is to encode point sets and disjoint holes only using triple orientations (see Section~\ref{sec:prelim}),
and then to use a SAT solver to disprove 
the existence of sets with certain properties (see Section~\ref{sec:satmodel}).

In the Final Remarks (Section~\ref{sec:final_remarks})
we outline how our SAT model can be adapted to tackle related questions on point sets.
For interior-disjoint holes, we show that every set of 15 points contains two interior-disjoint 5-holes.
Also it is remarkable,
that our SAT model can be used to prove $g(6)=17$ with significantly smaller computation time
than the original program from Szekeres and Peters~\cite{SzekeresPeters2006}\footnote{%
Szekeres and Peters  considered the problem of finding sets without 6-gons
also in the setting of triple-orientations (cf.\ Section~\ref{sec:prelim}),
and implemented a sophisticated exhaustive search technique in a classical program 
 (no SAT solvers involved).
}.
Last but not least, we also outline how SAT solvers can be used to 
count occurences of certain substructures (such as $k$-holes in point sets).

\section{Two Disjoint Holes}
\label{sec:results_two}

For two parameters, the value $h(k_1,k_2)$ has been determined 
for all $k_1,k_2 \le 5$, except for $h(5,5)$, 
by Hosono and Urabe~\cite{HOSONO200197,Hosono2005,Hosono2008} 
and by Bhattacharya and Das~\cite{BhattacharyaDas2011}.
Table~\ref{tab:2_disjoin_holes} summarizes the currently best bounds for two-parametric values.
The upper bounds were obtained via case distinctions 
and, to obtain values of $h(k,5)$, also the value $h(5)=10$ was utilized \cite{Harborth78}.
Lower bounds are witnessed by concrete examples of point sets.
It is also worth mentioning that all statements, which involve  11 or less points,
can be verified by checking the order type database of 11 points\footnote{%
The database of all combinatorially different sets of $n \le 10$ points is available online 
at~\cite{AichholzerOTDB} 
and requires roughly 550~MB of storage.
The database for $n=11$ requires about 100~GB of storage and is available on request.
For more information we refer to~\cite{Krasser03,aak-eotsp-01a,ak-aoten-06}.}.
However, for  $h(4,5)=12$~\cite{BhattacharyaDas2011} 
this database does not directly allow a direct proof.

Concerning the value $h(5,5)$, the best bounds are \mbox{$17 \le h(5,5) \le 19$}.
The lower bound $ h(5,5) \ge 17$ is witnessed by a set of 16 points 
with no two disjoint 5-holes (taken from Hosono and Urabe~\cite{Hosono2008}),
which is depicted Figure~\ref{fig:n16}. 
For the upper bound, Bhattacharya and Das~\cite{BhattacharyaDas2013}
used elaborate case distinctions to reveal more and more structural information
of point sets without (two disjoint) 5-holes to finally conclude that $h(5,5) \le 19$.

\begin{table}[htb]
\def\arraystretch{1.2}
	\centering
	\begin{tabular}{r|llll}
		&2	&3	&4	&5\\
	\hline
	2	&4	&5	&6	&10\\	
	3	&	&6	&7	&10\\
	4	&	&	&9	&12\\	
	5	&	&	&	&17*\\
	\end{tabular}
	\caption{Values of $h(k_1,k_2)$. 
	The entry marked with star (*) is new.}
	\label{tab:2_disjoin_holes}
\end{table}

\begin{figure}[htb]
\centering

\hfill
\begin{minipage}{0.55\textwidth}
\includegraphics[width=\textwidth]{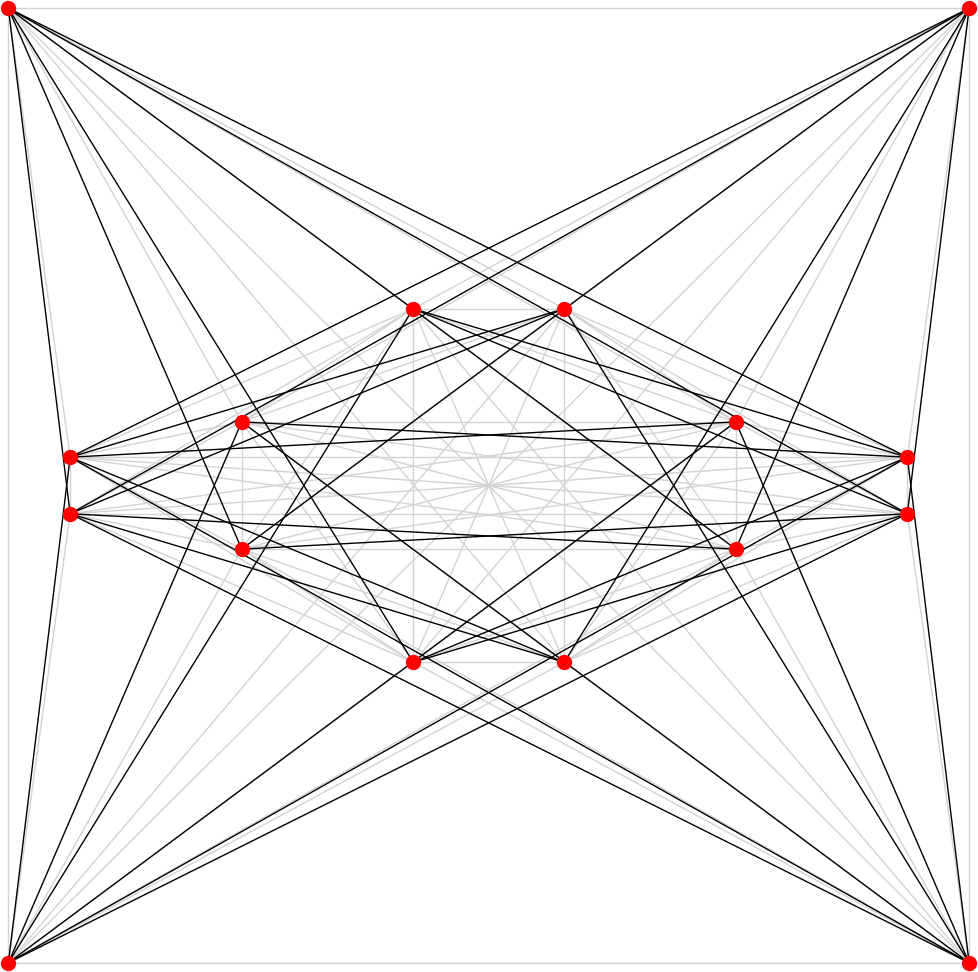}
\end{minipage}
\hspace{0.5cm}
\begin{minipage}{0.22\textwidth}
{
\small
\begin{verbatim}
  0   0
  0 270
280   0
280 270
 18 127
 18 143
262 127
262 143
 68 117
 68 153
212 117
212 153
118  85
118 185
162  85
162 185
\end{verbatim}
}
\end{minipage}
\hfill

\caption{A set of 16 points with no two disjoint $5$-holes. 
The coordinates are given on the right side.
This point set and the one by Hosono and Urabe~\cite[Figure~3]{Hosono2008} are of the same order type 
(order types are defined in Section~\ref{ssec:triple_orientations}).}
\label{fig:n16}
\end{figure}

\medskip
As our main result of this paper,
we determine the precise value of $h(5,5)$.
The proof is based on a SAT model which we later describe in Section~\ref{sec:satmodel}.
We remark that our SAT model can easily be adapted 
to also verify the other entries of Table~\ref{tab:2_disjoin_holes}.

\begin{theorem}[Computer-assisted]
\label{thm:h55}
Every set of 17 points contains two disjoint 5-holes, 
hence $h(5,5)=17$.
\end{theorem}

The computations for verifying Theorem~\ref{thm:h55}
take about two hours on a single 3~GHz CPU using a modern SAT solver such as 
glucose (version~4.0)\footnote{\url{http://www.labri.fr/perso/lsimon/glucose/}, see also~\cite{AudemardS09}} or 
picosat (version~965)\footnote{\url{http://fmv.jku.at/picosat/}, see also~\cite{Biere2008}}.
Moreover, we have verified the output of glucose and picosat with the proof checking tool
drat-trim\footnote{\url{http://cs.utexas.edu/~marijn/drat-trim}, see also~\cite{WetzlerHeuleHunt2014}}
(see Section~\ref{ssec:satmodel_unsat}).

\section{Three Disjoint Holes}
\label{sec:results_three}

For three parameters, 
most values $h(k_1,k_2,k_3)$ for $k_1,k_2,k_3 \le 4$  
and also the values $h(2,3,5)=11$ and $h(3,3,5)=12$ have been determined by Hosono and Urabe~\cite{Hosono2008} and by You and Wei~\cite{You2015}.
Tables~\ref{tab:3_disjoin_holes_k1_k2_4} and~\ref{tab:3_disjoin_holes_k1_k2_5} 
summarize the currently best known bounds for three-parametric values.
Again it is worth mentioning that all statements, which involve  11 or less points,
can be verified by checking the order type database of 11 points \cite{AichholzerOTDB}.

\begin{table}[htb]
\def\arraystretch{1.2}
\parbox[b]{.47\linewidth}{
	\centering
	\begin{tabular}{r|llll}
		&2	&3	&4	\\
	\hline
	2	&8	&9	&11	\\	
	3	&	&10	&12	\\
	4	&	&	&14	\\	
	\end{tabular}
	\caption{Values of $h(k_1,k_2,4)$.}
	\label{tab:3_disjoin_holes_k1_k2_4}
}
\hfill
\parbox[b]{.5\linewidth}{
	\centering
	\begin{tabular}{r|lllll}
		&2	&3	&4	&5	\\
	\hline
	2	&10 	&11	&11..14	&17*\\	
	3	&	&12	&13..14	&17..19*\\
	4	&	&	&15..17	&17..23*\\	
	5	&	&	&	&22*..27*\\	
	\end{tabular}
	\caption{Bounds for $h(k_1,k_2,5)$.}
	\label{tab:3_disjoin_holes_k1_k2_5}
}
\end{table}

The values $h(2,2,4)$, $h(3,3,4)$, and $h(2,4,4)$ 
have not been explicitly stated in literature.
However, the former two can be derived directly from other values as follows:
\begin{align*}
8 = 2+2+4 \le h(2,2,4) &\le 2+h(2,4)= 8\\
10 = 3+3+4 \le h(3,3,4) & \le 3+h(3,4) = 10
\end{align*}
To determine the value $h(2,4,4)=11$,
observe that $h(2,4,4) \le 2+h(4,4) = 11$ clearly holds.
Equality is witnessed by
the \emph{double circle}  
with 10 points 
(cf.~Figure~\ref{fig:double_circle_10}).
This statement can be verified by computer or as follows:
First, observe that no 4-hole contains 
two consecutive extremal points,
thus every 4-hole contains at most two exterior points.
Now consider two disjoint 4-holes.
Since not both 4-holes can contain two extremal points,
one of them
contains  two exterior points
while the other one contains  one exterior point.
As illustrated in Figure~\ref{fig:double_circle_10},
this configuration is unique up to symmetry and
does not allow any further disjoint 2-hole.
This completes the argument.

\begin{figure}[htb]
  \centering
    \includegraphics{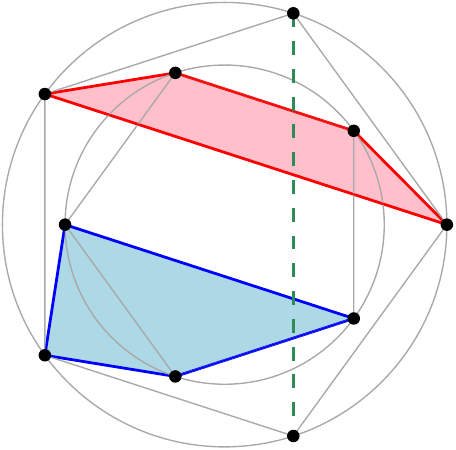}
  \caption{The double circle on 10 points witnesses $h(2,4,4)>10$. }
  \label{fig:double_circle_10}
\end{figure}

Also we could not find the value $h(2,2,5)$ in literature,
however, using a SAT instance similar to the one for Theorem~\ref{thm:h55}
one can also easily verify that $h(2,2,5) \le 10$,
and equality follows from $h(5)=10$~\cite{Harborth78}.
One can also use the order type database \cite{AichholzerOTDB} 
to verify $h(2,2,5) \le 10$.

\medskip
We now use
Theorem~\ref{thm:h55} to derive new bounds 
on the value $h(k,5,5)$ for $k =2,3,4,5$.

\begin{corollary}
We have 
\begin{alignat*}{4}
          && h(2,5,5) &\,=\,& 17,\\
17 &\,\le\,& h(3,5,5) &\,\le \,& 19,\\
17 &\,\le\,& h(4,5,5) &\,\le \,& 23,\\
22 &\,\le\,& h(5,5,5) &\,\le \,& 27.
\end{alignat*}
\end{corollary}

\begin{proof} 
To show $h(2,5,5) \le 17$,
observe that, due to Theorem~\ref{thm:h55},
every set of 17 points contains two disjoint 5 holes
that are separated by a line~$\ell$. 
By the pigeonhole principle there are at least 9 points on one of the two sides of such a separating line~$\ell$.
It is implied by $h(2,2,5)=10$ that
every set of 9 points with a 5-hole also contains a 2-hole which is
disjoint from the 5-hole.
This completes the argument.
We remark that one can also use the order type database of 9 points to verify this statement.

To show $h(3,5,5) \le 2\cdot h(3,5)-1 =19$,
observe that, due to Theorem~\ref{thm:h55},
every set of 19 points contains two disjoint 5 holes
that are separated by a line~$\ell$. 
Now there are at least 10 points on one side of such a separating line~$\ell$,
and since $h(3,5)=10$, there is a 3-hole and a 5-hole that are disjoint 
on that particular side. 

An analogous argument shows $h(4,5,5) \le 2\cdot h(4,5)-1 = 23$.

The set of 21 points depicted in Figure~\ref{fig:n21} witnesses $h(5,5,5) > 21$
(can be easily verified by computer),
while $h(5,5,5) \le h(5)+h(5,5) = 27$.
We remark that this point set was found using local search techniques,
implemented in our framework \emph{pyotlib}\footnote{%
The ``\textbf{py}thon \textbf{o}rder \textbf{t}ype \textbf{lib}rary'' was initiated during 
the Bachelor's studies of the author~\cite{scheucher2014} 
and provides many features to work with (abstract) order types
such as 
local search techniques, 
realization or proving non-realizability of abstract order types, 
coordinate minimization and ``beautification'' for nicer visualizations.
For more information, please consult the author.}.
The key idea was to start with an arbitrary set of 21 points
and to move points around until 
the number of triples of disjoint 5-holes becomes zero.
\end{proof}

\begin{figure}[htb]
\centering
\begin{minipage}{0.7\textwidth}
\includegraphics[width=\textwidth]{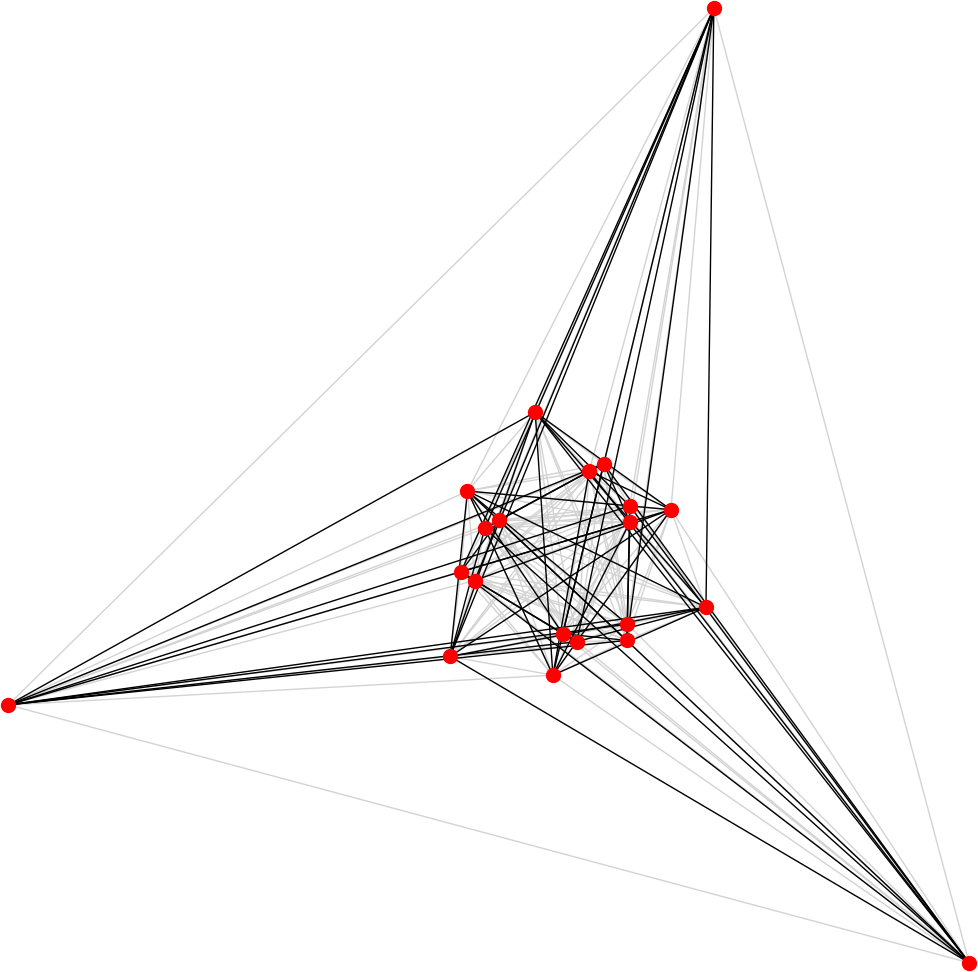}
\end{minipage}
\hspace{0.5cm}
\begin{minipage}{0.2\textwidth}
{
\small
\begin{verbatim}
     0 161014
437034 595949
326347 343801
284425 294548
368806 311583
359850 306967
303825 276373
295136 271265
384946 285229
410465 282863
385025 275150
280383 244110
288858 238662
432159 221931
383508 211334
343366 205440
352134 200469
273710 191231
383027 201270
337326 179552
595182      0
\end{verbatim}
}
\end{minipage}
\caption{A set of 21 points with no three disjoint $5$-holes.
The coordinates are given on the right side.}
\label{fig:n21}
\end{figure}

\section{Many Disjoint Holes}
\label{sec:results_many}

As introduced by Hosono and Urabe~\cite{HOSONO200197,Hosono2008}, we use the following notation:
Given positive integers $k$ and $n$, let $F_k(n)$ denote the maximum number of pairwise disjoint $k$-holes that 
can be found in every set of $n$ points,
that is,
\[
F_k(n) := \max (\{0\} \cup \{t \in \mathbb{N} \colon h(k;t) \le n \}) 
\quad  
\text {with}
\quad 
h(k;t) := h(\overbrace{k,k,\ldots,k}^{\text{$t$ parameters}}).
\]
In the following, we revise and further improve results by
Hosono and Urabe~\cite{HOSONO200197,Hosono2008}
and by
B{\'a}r{\'a}ny and K{\'a}rolyi~\cite{BaranyiKarolyi2001}.
The currently best bounds are the following:
\begin{alignat*}{5}
&&F_k(n)&\, = \,& \lfloor n/k \rfloor &\quad \text{ for $k=1,2,3$} \\
3n/13 - o(n)   &\,\le\,& F_4(n)&\, \le \,& n/4\\
2n/17 -O(1) &\,\le \,& F_5(n)&\, \le \,& n/6+O(1)\\
n/h(6)-O(1)&\, \le \,&F_6(n)&\, \le \,& n/8+O(1)\\
&&F_k(n)& \,=\, & 0 &\quad \text{ for $k \ge 7$.}
\end{alignat*}

Concerning the lower bounds, 
Theorem~\ref{thm:h55} clearly implies $F_5(n) \ge \lfloor 2n/17 \rfloor $.
In the following, we outline the proof of   $ F_4(n) \ge 3n/13 - o(n) $.
Hosono and Urabe~\cite{HOSONO200197} showed that $F_4(n) \ge (3n-1)/13$ 
holds for an infinite sequence of integers~$n$, which we denote by~$N$.
Let $T := \{F_4(n) : n \in N\}$.
Note that $T$ is also infinite.
From the definition of $F_4(n)$ and the $F_4(n) \ge (3n-1)/13$ bound for $n\in N$,
we conclude that
\[ 
h(4;F_4(n))  \le n \le \frac{13F_4(n)+1}{3} 
\]
holds for every $n \in N$,
and hence $h(4;t) \le \frac{13t+1}{3}$ for every $t \in T$.
Since $h(4;s+t)$ is \emph{subadditive},
that is $h(4;s+t) \le h(4;s) + h(4;t)$,
Fekete's subadditivity lemma
(see for example~\cite[Chapter 14.5]{schrijver-book})
asserts 
\[
 \lim_{t \to \infty}\frac{h(4;t)}{t} 
\ =\ \inf_{t \in \mathbb{N} } \frac{h(4;t)}{t}
\ \le \ \inf_{t \in T } \frac{h(4;t)}{t} = \frac{13}{3}.
\]
Hence,
$h(4;t) \le \frac{13t}{3} (1+o(1))$ and
we conclude $F_4(n) \ge \frac{3n}{13} (1-o(1))$.

\medskip

Concerning the upper bounds,
B{\'a}r{\'a}ny and K{\'a}rolyi \cite{BaranyiKarolyi2001} remarked
that no nontrivial upper bound is known for $F_4 (n)$ in general.
They mentioned $F_5(n)<\frac{n}{6}$ 
without an explicit construction 
but we only know a construction for
$F_5(n) \leq \frac{n+1}{6}$ (Gyula K{\'a}rolyi, personal communication).
In the following, we give the construction
for $F_5(n) \leq \frac{n+1}{6}$ and $F_6(n) \leq \frac{n+1}{8}$.

\medskip

We assume that $n$ is even, and show $F_5(n) \le \frac{n}{6}$ and $F_6(n) \leq \frac{n}{8}$. 
(For $n$ odd, we can then conclude  
$ F_5(n) \le F_5(n+1) \le \frac{n+1}{6}$ and $ F_6(n) \le F_6(n+1) \le \frac{n+1}{8}$ from monotonicity.)
As illustrated in Figure~\ref{fig:F5_construction}, 
we take the
vertices of a regular $\frac{n}{2}$-gon (extremal points) plus the vertices
of a slightly shrinked copy of the regular $\frac{n}{2}$-gon with the same center (inner points).
If $\frac{n}{2}$ is even, we slightly rotate the inner points
around the center to maintain general position.
This way we get $\frac{n}{2}$ pairs of very close points.

\begin{figure}[htb]
  \centering
    \includegraphics{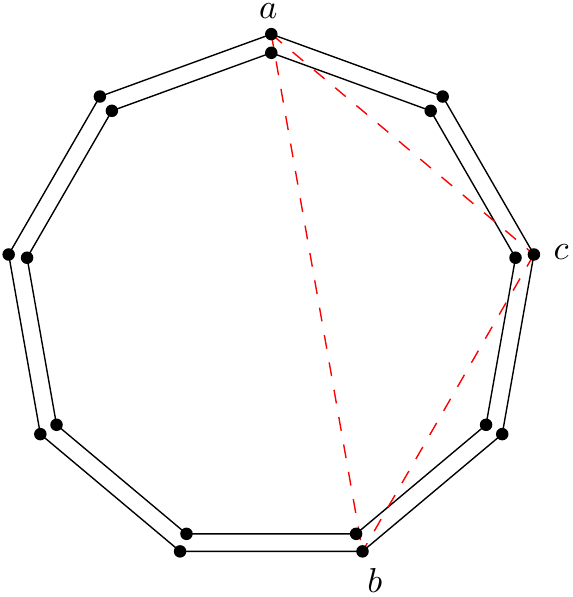}
  \caption{An illustration of the construction with $n=18$ points.}
  \label{fig:F5_construction}
\end{figure}

We now show that no three extremal points form a 3-hole.
Suppose towards a contradiction 
that three extremal points $a,b,c$ span an empty triangle~$\triangle$.
By symmetry we may assume 
that $a$ is the top-most vertex of the outer $\frac{n}{2}$-gon,
and that $b$ lies below~$c$.
Both points $b$ and $c$ must either lie to the left 
or to the right of $a$, 
as otherwise the inner partner of~$a$ would lie inside of~$\triangle$.
Now, however, the inner partner of~$c$ lies in~$\triangle$ 
(cf.\ Figure~\ref{fig:F5_construction}) and hence $\triangle$ cannot be empty. A contradiction.

We conclude that every 5-hole (6-hole) 
is incident 
to at most~2 extremal points
and hence to at least 3 (4) inner points.
Therefore, 
at most~$\frac{2}{3}$ ($\frac{2}{4}$) of the exterior points 
can be covered by disjoint 5-holes (6-holes),
and we conclude $F_5(n) \le \frac{n}{6}$ and $F_6(n) \leq  \frac{n}{8}$ for $n$ even.

\section{Encoding with Triple Orientations}
\label{sec:prelim}

In this section
we describe how point sets and disjoint holes 
can be encoded only using triple orientations.
This combinatorial description allows us to 
get rid of the actual point coordinates
and to only consider a discrete parameter-space.
This is essential for our SAT model of the problem.

\subsection{Triple Orientations}
\label{ssec:triple_orientations}

Given a set of points $S=\{s_1,\ldots,s_n\}$ with $s_i = (x_i,y_i)$,
we say that the triple $(a,b,c)$ is \emph{positively (negatively) oriented} 
if 
\[
\chi_{abc} := \sgn \det \begin{pmatrix}
                         1  & 1  & 1  \\
                         x_a& x_b& x_c\\
                         y_a& y_b& y_c
                        \end{pmatrix} \in \{-1,0,+1\}
\]
is positive (negative)\footnote{%
 The letter $\chi$ is commonly used in literature 
 to denote triple orientations as the word ``chirality''
 is derived from the Greek word for ``hand''.}.
 Equivalently, the triple $(a,b,c)$ is positively (negatively) oriented 
 if the point $c$ lies to the left (right) of the \emph{directed} line~$\overrightarrow{ab}$.
 Figure~\ref{fig:triple_orientations} gives an illustration.
 Note that $\chi_{abc}=0$ indicates collinear points, 
 in particular, $\chi_{aaa}=\chi_{aab}=\chi_{aba}=\chi_{baa}=0$.
 
\begin{figure}[htb]

	\begin{subfigure}[t]{.45\textwidth}
		\centering
		\includegraphics[page=1]{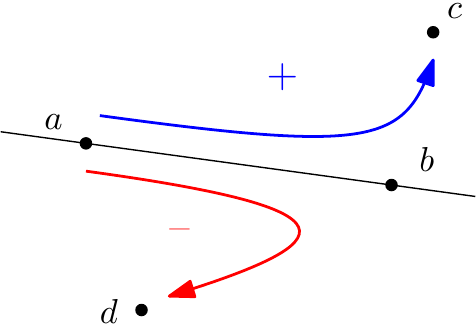}
		\caption{}
		\label{fig:triple_orientations}  
	\end{subfigure}
	\hfill
	\begin{subfigure}[t]{.45\textwidth}
		\centering
		\includegraphics[page=2]{triple_orientations}
		\caption{}
		\label{fig:separation}  
	\end{subfigure}

  \caption{An illustration of \subref{fig:triple_orientations}~triple-orientations and \subref{fig:separation}~an $A$-$B$-separation.}
\end{figure}

It is easy to see, 
that convexity is a combinatorial rather than a geometric property 
since $k$-gons
can be described only by the relative position of the points:
If the points $s_1,\ldots,s_k$ are the vertices of a convex polygon
(ordered along the boundary),
then, for every $i=1,\ldots,k$, 
the cyclic order of the other points around $s_i$ is 
$s_{i+1},s_{i+2},\ldots,s_{i-1}$ (indices modulo~$k$).
Similarly, one can also describe containment (and thus $k$-holes) 
only using relative positions: 
A~point $s_0$ lies inside a convex polygon with vertices $s_1,\ldots,s_k$ (ordered along the boundary) 
if, for every $i=1,\ldots,k$, 
the line $\overline{s_is_{i+1}}$ does not separate $s_0$ 
from the remaining vertices $\{s_1,\ldots,s_k\} \setminus \{s_i,s_{i+1}\}$ (indices modulo~$k$).

To observe that the disjointness of two point sets
can be described solely using triple orientations,
suppose that a line $\ell$ separates point sets $A$ and~$B$.
Then, for example by rotating $\ell$, we can find another 
line $\ell'$ that contains a point $a\in A$ and a point $b\in B$
and separates $A\setminus \{a\}$ and $B\setminus \{b\}$.
In particular,
we have $\chi_{aba'} \le 0$ for all $a' \in A$
and $\chi_{abb'} \ge 0$ for all $b' \in B$, or the other way round.
Figure~\ref{fig:separation} gives an illustration.
Altogether, the existence of disjoint holes can be described solely using triple orientations.
 
Note that, even though
there are uncountable possibilities to choose 
$n$ points from the Euclidean plane for fixed~$n \in \mathbb{N}$,
there are only finitely many equivalence classes of point sets 
when point sets inducing the same orientation triples are considered equal.
As introduced by Goodman and Pollack~\cite{Goodman:1983:doi:10.1137/0212032},
these equivalence classes 
(sometimes also with unlabeled points) 
are called \emph{order types}.

\subsection{An Abstraction of Point Sets}
\label{ssec:abstract_point_sets}

Consider a point set $S=\{s_1,\ldots,s_n\}$ 
where $s_1,\ldots,s_n$ have increasing $x$-coordinates.
Using the \emph{unit paraboloid duality transformation},
which maps a point $s=(a,b)$ to the line $s^*:y=2ax-b$,
we obtain the arrangement of dual lines $S^* = \{s_1^*,\ldots,s_n^*\}$,
where the dual lines $s_1^*,\ldots,s_n^*$ have increasing slopes.
By the increasing $x$-coordinates
and the properties of the unit paraboloid duality 
(cf.\ \cite[Chapter~6.5]{ORourke1994_book} or \cite[Chapter~1.4]{Edelsbrunner1987_book}), 
for every three points $s_i,s_j,s_k$ with $i<j<k$ 
the following three statements are equivalent:
\begin{enumerate}[(i)]
\item 
The points $s_i,s_j,s_k$ are positively oriented.
\item
The point $s_k$ lies above the line $\overline{s_is_j}$.
\item 
The intersection-point of the two lines $s_i^*$ and $s_j^*$ lies above the line $s_k^*$.
\end{enumerate}
Due to Felsner and Weil~\cite{FelsnerWeil2001} 
(see also Balko, Fulek, and Kyn{\v{c}}l~\cite{BalkoFulekKyncl2015}),
for every four points $s_i,s_j,s_k,s_l$ with {$i<j<k<l$},
the sequence 
\[
\chi_{ijk},\ \chi_{ijl},\ \chi_{ikl},\ \chi_{jkl}
\] 
(index-triples are in lexicographic order)
changes its sign at most once. 
These conditions are the \emph{signotope axioms}.
It is worth mentioning that the signotope axioms were also used 
in the computer-assisted proof for $g(6)=17$ by Szekeres and Peters \cite{SzekeresPeters2006}
and also later by Balko and Valtr~\cite{BalkoValtr2017},
who refuted a natural strengthening of the Erd\H{o}s--Szekeres conjecture 
introduced by Szekeres and Peters.

The signotope axioms are necessary conditions 
but not sufficient to axiomize point sets:
As described above, 
every point set $S$ induces a mapping $\chi_S: {n \choose 3} \to \{+,-\}$, 
which fulfills the signotope axioms.
However, there exist mappings $\chi_S: {n \choose 3} \to \{+,-\}$ fulfilling the signotopes axioms
-- such mappings are called
\emph{signotopes},
\emph{abstract point sets},
\emph{abstract order types}, and
\emph{abstract oriented matroids (of rank 3)}
--
which are \textbf{not} induced by any point set,
and in fact, deciding
whether an abstract point set has a realizing point set 
is known to be $\exists \mathbb{R}$-complete.
For more information about realizability we refer to~\cite{FelsnerGoodman2016}.

\subsection{Increasing Coordinates and Cyclic Order}

In the following, we see why
we can assume 
that in every point set $S=\{s_1,\ldots,s_n\}$ with $s_i=(x_i,y_i)$ the following three conditions hold:
\begin{itemize}
\item 
the points $s_1,\ldots,s_n$ have increasing $x$-coordinates (i.e.,~$x_1 \le \ldots \le x_n$)
\item
in particular, $s_1$~is the leftmost point, and
\item 
the points $s_2,\ldots,s_n$ are sorted around $s_1$.
\end{itemize}

When modeling a computer program,
one can use these constraints (which do not affect the output of the program)
to restrict the search space and to possibly get a speedup. 
This idea, however, is not new and was already used 
for the generation of the \emph{order type database}, 
which provides a complete list of all order types of up to $11$ points~\cite{Krasser03,aak-eotsp-01a,ak-aoten-06}.

\begin{lemma}
\label{lemma:increasing_coordinates}
Let $S=\{s_1,\ldots,s_n\}$ 
be a point set where
$s_1$ is extremal and $s_2,\ldots,s_n$ are sorted around $s_1$.
Then there is a point set 
$\tilde{S}=\{\tilde{s_1},\ldots,\tilde{s_n}\}$ of the same order type as~$S$
(in particular, $\tilde{s_2},\ldots,\tilde{s_n}$ are sorted around~$\tilde{s_1}$)
such that the points $\tilde{s_1},\ldots,\tilde{s_n}$ have increasing $x$-coordinates.
\end{lemma}

\begin{proof}
We can apply an appropriate affine-linear transformation onto~$S$ 
so that $s_1=(0,0)$ and $x_i,y_i > 0$ holds for $i \ge 2$:
First we apply a translation so that $s_1$ lies in the origin,
then we rotate such that all points $s_2,\ldots,s_n$ have positive $x$-coordinate,
and finally we apply a shearing transformation so that $s_2,\ldots,s_n$ have positive $y$-coordinate as well.

We have that
$x_i/y_i$ is increasing for $i \ge 2$ since  $s_2,\ldots,s_n$ are sorted around $s_1$.
Since $S$ is in general position, there is an $\varepsilon > 0$ such that
$S$ and $S' := \{(0,\varepsilon)\} \cup \{s_2,\ldots,s_n\}$ are of the same order type.  
We apply the projective transformation
\mbox{$(x,y) \mapsto (\nicefrac{x}{y},\nicefrac{-1}{y})$}
to $S'$ to obtain~$\tilde{S}$. 
By the multilinearity of the determinant, we obtain
\[
\det 
\begin{pmatrix}
1 & 1 & 1\\
x_i & x_j & x_k\\
y_i & y_j & y_k\\
\end{pmatrix}
=
y_i \cdot y_j \cdot y_k \cdot 
\det 
\begin{pmatrix}
1 & 1 & 1\\
 \nicefrac{x_i}{y_i} &  \nicefrac{x_j}{y_j} &  \nicefrac{x_k}{y_k}\\
\nicefrac{-1}{y_i} & \nicefrac{-1}{y_j} & \nicefrac{-1}{y_k}\\
\end{pmatrix}.
\]
Since the points in $S'$ have positive $y$-coordinates,
$S'$ and $\tilde{S}$ have the same triple orientations.
Moreover, as $\tilde{x_i}=\nicefrac{x_i'}{y_i'}$ is increasing for $i \ge 1$, 
the set $\tilde{S}$ fulfills all desired properties.
\end{proof}

It is worth to mention that the transformation
$(x,y) \mapsto (\nicefrac{x}{y},\nicefrac{-1}{y})$
is the concatenation of the 
(inverse of the) 
\emph{unit paraboloid duality transformation}
and
\emph{unit circle duality transformation}
which -- under the given conditions -- 
both preserve the triple orientations
(see e.g.~\cite[Chapters~1.3 and~2.2]{Krasser03}).

\section{SAT Model}
\label{sec:satmodel}

In this section we describe the SAT model
that we use to prove Theorem~\ref{thm:h55}.
The basic idea of the proof is to assume towards a contradiction
that a point set $S=\{s_1,\ldots,s_{17}\}$ with no two disjoint 5-holes exists.
We formulate a SAT instance, where Boolean variables indicate whether 
triples are positively or negatively oriented
and clauses encode the necessary conditions introduced in Section~\ref{sec:prelim}.
Using a SAT solver we verify that the SAT instance has no solution
and conclude that the point set $S$ does not exist.
This contradiction then completes the proof of Theorem~\ref{thm:h55}.

\medskip
It is folklore that satisfiability is \NP-hard in general,
thus it is challenging for SAT solvers to 
terminate in reasonable time for certain SAT instances.
We now highlight the
two crucial parts of our SAT model, 
which are indeed necessary for reasonable computation times:
First, due to Lemma~\ref{lemma:increasing_coordinates},
we can assume that
the points are sorted from left to right and also around the first point~$s_1$.
Second, we teach the solver that every set of 10 consecutive points 
gives a 5-hole, that is, $h(5)=10$~\cite{Harborth78}.
By dropping either of these two constraints 
(which only give additional information to the solver and do not affect the solution space),
none of the tested SAT solvers terminated within days.

\medskip

In the following, we give a detailed description of our SAT model.
For the sake of readability, we refer to points also by their indices.
Moreover, we use the relation ``$a<b$'' simultaneously 
to indicate a larger index, a larger $x$-coordinate, 
and the later occurence in the cyclic order around~$s_1$.

\subsection{A Detailed Description}

\paragraph{(1) Alternating axioms}

For every triple $(a,b,c)$, we introduce the variable $O_{a,b,c}$ to indicate
whether the triple $(a,b,c)$ is positively oriented.
Since we have that
\[
\chi_{a,b,c} = \chi_{b,c,a} = \chi_{c,a,b} = -\chi_{b,a,c} = -\chi_{a,c,b} = -\chi_{c,b,a}, 
\]
we formulate clauses to assert
\[
O_{a,b,c} =
O_{b,c,a} = 
O_{c,a,b} 
\neq
O_{b,a,c} =
O_{a,c,b} =
O_{c,b,a}
\]
by using the fact
$
A = B \Longleftrightarrow ( \neg A \vee B) \wedge (  A \vee \neg B) ,
$
and
$
A \neq B \Longleftrightarrow ( A \vee B) \wedge (  \neg A \vee \neg B).
$

\paragraph{(2) Signotope Axioms}

As described in Section~\ref{ssec:abstract_point_sets},
for every 4-tuple $a<b<c<d$,
the sequence 
\[ 
\chi_{abc},\ \chi_{abd},\ \chi_{acd},\ \chi_{bcd}
\]
changes its sign at most once.
Formally,
to forbid other sign-patterns (that is, ``$-+-$'' and ``$+-+$''),
we add the constraints
\begin{alignat*}{15}
     & O_{a,b,c} &\ \vee\ & 
\neg & O_{a,b,d} &\ \vee\ & 
     & O_{a,c,d} 
\quad
\quad
\quad
\quad
\neg & O_{a,b,c} &\ \vee\ & 
     & O_{a,b,d} &\ \vee\ & 
\neg & O_{a,c,d} 
\\
     & O_{a,b,c} &\ \vee\ & 
\neg & O_{a,b,d} &\ \vee\ & 
     & O_{b,c,d} 
\quad
\quad
\quad
\quad
\neg & O_{a,b,c} &\ \vee\ & 
     & O_{a,b,d} &\ \vee\ & 
\neg & O_{b,c,d} 
\\
     & O_{a,b,c} &\ \vee\ & 
\neg & O_{a,c,d} &\ \vee\ & 
     & O_{b,c,d} 
\quad
\quad
\quad
\quad
\neg & O_{a,b,c} &\ \vee\ & 
     & O_{a,c,d} &\ \vee\ & 
\neg & O_{b,c,d} 
\\
     & O_{a,b,d} &\ \vee\ & 
\neg & O_{a,c,d} &\ \vee\ & 
     & O_{b,c,d} 
\quad
\quad
\quad
\quad
\neg & O_{a,b,d} &\ \vee\ & 
     & O_{a,c,d} &\ \vee\ & 
\neg & O_{b,c,d} 
\\
\end{alignat*}

\paragraph{(3) Sorted around first point}
 
Since the points $s_1,\ldots,s_n$ are sorted from left to right
and also around the first point~$s_1$,
we have that all triples $(1,a,b)$ are positively oriented for indices $1 < a < b$.

\paragraph{(4) Bounding segments}

For a 4-tuple $a,b,c,d$,
we introduce the auxiliary variable
$E_{a,b;c,d}$ to indicate whether
the segment $ab$ spanned by $a$ and $b$ 
bounds the convex hull of $\{a,b,c,d\}$.
Since the segment $ab$ bounds $\conv(\{a,b,c,d\})$
if and only if $c$ and $d$ lie on the same side of the line~$\overline{ab}$,
we add the constraints
\begin{alignat*}{6}
\neg E_{a,b;c,d} &\ \vee\ &      &O_{a,b,c} &\ \vee\ & \neg &O_{a,b,d}, \\
\neg E_{a,b;c,d} &\ \vee\ & \neg &O_{a,b,c} &\ \vee\ &      &O_{a,b,d}, \\
     E_{a,b;c,d} &\ \vee\ &      &O_{a,b,c} &\ \vee\ &      &O_{a,b,d}, \\
     E_{a,b;c,d} &\ \vee\ & \neg &O_{a,b,c} &\ \vee\ & \neg &O_{a,b,d} .
\end{alignat*}

\paragraph{(5) 4-Gons and containments}

For every 4-tuple $a<b<c<d$, we introduce
the auxiliary variable $G^4_{a,b,c,d}$
to indicate whether the points $\{a,b,c,d\}$ form a 4-gon.
Moreover we introduce the auxiliary variable $I_{i;a,b,c}$ for every 4-tuple $a,b,c,i$ with $a<b<c$ and $a<i<c$
to indicate whether the point $i$ lies inside the triangular convex hull of $\{a,b,c\}$.

Four points $a<b<c<d$, sorted from left to right,
form a 4-gon if and only if 
both segments $ab$ and $cd$ bound $\conv(\{a,b,c,d\})$.
Moreover, if $\{a,b,c,d\}$ does not form a \mbox{4-gon}, 
then either $b$ lie inside the triangular convex hull $\conv(\{a,c,d\})$ 
or $c$ lies inside $\conv(\{a,b,d\})$.
Pause to note that $a$ and $d$ are the left- and rightmost points, respectively, 
and that not both points $b$ and $c$ can lie in the interior of $\conv(\{a,b,c,d\})$.
Formally, we assert 
\begin{alignat*}{4}
G^4_{a,b,c,d} &\ =\ &      E_{a,b;c,d} &\ \wedge\ &      E_{c,d;a,b},\\
  I_{b;a,c,d} &\ =\ & \neg E_{a,b;c,d} &\ \wedge\ &      E_{c,d;a,b},\\
  I_{c;a,b,d} &\ =\ &      E_{a,b;c,d} &\ \wedge\ & \neg E_{c,d;a,b}.
\end{alignat*}

\paragraph{(6) 3-Holes}

For every triple of points $a<b<c$, we introduce the auxiliary variable
$H^3_{a,b,c}$ to indicate whether the points $\{a,b,c\}$ form a 3-hole.
Since three points $a<b<c$ form a 3-hole if and only if 
every other point $i$ with $a<i<c$ lies outside the triangular convex hull $\conv(\{a,b,c\})$,
we add the constraint
\[
H^3_{a,b,c} = \bigwedge_{i \in S \setminus \{a,b,c\}} \neg I_{i;a,b,c}.
\]

\paragraph{(7) 5-Holes}

For every 5-tuple $X=\{a,b,c,d,e\}$ with $a<b<c<d<e$, 
we introduce the auxiliary variable $H^5_X$ to indicate
that the points from $X$ form a 5-hole.
It is easy to see that the points from $X$ form a 5-hole if and only if
every 4-tuple $Y \in {X \choose 4}$ forms a 4-gon
and if every triple $Y \in {X \choose 3}$ forms a 3-hole.
Therefore, we add the constraint
\[
H^5_{X} = 
\biggl(
\ 
\bigwedge_{Y \in {X \choose 4}}  G^4_{Y} 
\ 
\biggr)
\wedge 
\biggl(
\ 
\bigwedge_{Y \in {X \choose 3}}  H^3_{Y} 
\ 
\biggr).
\]

\paragraph{(8) Forbid disjoint 5-holes}

If there were two disjoint 5-holes $X_1$ and~$X_2$ in our point set~$S$,
then -- as discussed in Section~\ref{sec:prelim} --
we could find two points $a \in X_1$ and $b \in X_2$
such that the line~$\overline{ab}$ separates 
$X_1 \setminus \{a\}$ and $X_2 \setminus \{b\}$ (cf.~Figure~\ref{fig:separation}) -- 
and this is what we have to forbid in our SAT model.
Hence, for every pair of two points $a,b$ we introduce the variables
\begin{itemize}
 \item 
 $L_{a,b}$ to indicate that
 there exists a 5-hole $X$ containing the point $a$
 that lies to the left of the directed line~$\overrightarrow{ab}$,
 that is, the triple $(a,b,x)$ is positively oriented for every $x \in X \setminus \{a\}$, and
 \item 
 $R_{a,b}$ to indicate that
 there exists a 5-hole $X$ containing the point $b$ 
 that lies to the right of the directed line~$\overrightarrow{ab}$,
 that is, the triple $(a,b,x)$ is negatively oriented for every $x \in X \setminus \{b\}$.
\end{itemize}
For every 5-tuple $X$ with $a \in X$ and $b \not\in X$ we assert
\[
L_{a,b} 
\ 
\vee 
\  
\neg H^5_X 
\ 
\vee
\ 
\biggl(
\ 
\bigvee_{c \in X \setminus \{a\}} \neg O_{a,b,c} 
\ 
\biggr),
\]
and for every 5-tuple $X$ with $a \not\in X$ and $b \in X$ we assert
\[
R_{a,b} 
\ 
\vee 
\  
\neg H^5_X 
\ 
\vee
\ 
\biggl(
\ 
\bigvee_{c \in X \setminus \{b\}} O_{a,b,c} 
\ 
\biggr).
\]
Now we forbid that there are 5-holes on both sides of the line $\overline{ab}$ by asserting
\[
 \neg L_{a,b} \vee \neg R_{a,b}.
\]

\paragraph{(9) Harborth's result}

Harborth~\cite{Harborth78} has shown that every set of 10 points 
gives a 5-hole\footnote{%
$h(5)=10$ can also be verified by slightly adapting the described SAT model},
that is, $h(5)=10$.
Since Harborth's result applies to each set of 10 consecutive points of $S$, 
we can teach the SAT solver that
\begin{itemize}
 \item for every $i=1,\ldots,8$, there is a 5-hole $X$ with $X \subset \{ i,\ldots, i+9\}$. 
\end{itemize}
Moreover, if there is a 5-hole $X_1$ in the set $\{1,\ldots,7\}$, 
then there is another 5-hole $X_2$ in the set $\{8,\ldots,17\}$.
Analogously, if there is a 5-hole $X_1$ in the set $\{11,\ldots,17\}$, 
then there is another 5-hole $X_2$ in the set $\{1,\ldots,10\}$.
Therefore, we can teach the SAT solver that
\begin{itemize}
 \item there is no 5-hole $X$ with $X \subset \{ 1,\ldots, 7\}$, and 
 \item there is no 5-hole $X$ with $X \subset \{11,\ldots,17\}$.
\end{itemize}

\medskip

We remark that the 
obtained SAT instance has 825\,689 constraints in 23\,392 variables,
and that the dominating parts is~(8).
The source code of our python program
which creates the instance is available as supplemental data and
on our website~\cite{website_disjoint_holes}.

\subsection{Unsatisfiability and Verification}
\label{ssec:satmodel_unsat}

Having the satisfiability instance generated, we used the following command 
to create an unsatisfiability certificate:
\begin{verbatim}
  glucose instance.cnf -certified -certified-output=proof.out
\end{verbatim}
The certificate cerated by glucose
was then verified using the proof checking tool
drat-trim by the following command:
\begin{verbatim}
  drat-trim instance.cnf proof.out
\end{verbatim}
The execution of each of the two commands (glucose and drat-trim), 
took about 2 hours and 
the certificate used about 3.1~GB of disk space.

We have also used picosat to prove unsatisfiability:
\begin{verbatim}
  picosat instance.cnf -R proof.out
\end{verbatim}
This command ran for about 6 hours
and created a certificate of size about 2.1~GB.
The verification of the certificate\footnote{%
In our experiments, picosat wrote a comment ``\%RUPD32 ...'' as first line in the RUP file. 
This line had to be removed manually to make the file  parsable for drat-trim.
}  using drat-trim took about 9~hours.

\section{Final Remarks}
\label{sec:final_remarks}

In (8), we have introduced the variable
$L_{a,b}$ to indicate that
there exists a 5-hole $X$ containing the point $a$
that lies to the left of the directed line~$\overrightarrow{ab}$.
By relaxing this to 
``\ldots there exists a 5-hole $X$, possibly containing the point $a$, \ldots''
and analogously for $R_{a,b}$, 
the computation time reduces by factor of roughly~2 
while the number of clauses raises by a factor of~$n$.
The solution space, however, remains unaffected.

\medskip
As pointed out by the anonymous reviewers,
the constraints
$
H^5_{X} = 
(
\bigwedge_{Y \in {X \choose 4}}  G^4_{Y} 
)
\wedge 
(
\bigwedge_{Y \in {X \choose 3}}  H^3_{Y} 
)
$
from (7) are equivalent to
$
H^5_{X} = 
\bigwedge_{Y \in {X \choose 3}}  H^3_{Y} 
$.
Replacing (7) by this simplified expression
 further makes the auxiliary variables $G^4_{a,b,c,d}$ and the first part of (5) obsolete.
However, since the described replacement 
did not show any effect on the running time of the solvers
and since we discuss $g(6)=17$ and the {Classical Erd\H{o}s--Szekeres} problem below,
we decided to keep (7), (5), and the auxiliary variables $G^4_{a,b,c,d}$.

\medskip

\paragraph{Multi-parametric Values:}
To determine multi-parametric values such as $h(5,5,5)$,
one can formulate a SAT instance as follows:
Three 5-holes $X_1,X_2,X_3$ 
are pairwise disjoint if 
there is a line $\ell_{ij}$ 
for every pair $X_i,X_j$ 
that separates $X_i$ and~$X_j$. 
By introducing auxiliary variables $Y_{i,j}$ 
for every pair of 5-tuples $X_i,Y_i$ 
to indicate whether $X_i$ and $X_j$ are disjoint 5-holes,
one can formulate an instance in $\Theta(n^{10})$ variables with $\Theta(n^{15})$ constraints.
However, since this formulation is quite space consuming, 
a more compact formulation might be of interest.

\paragraph{Interior-disjoint Holes:}
Besides disjoint holes, also the variant of interior-disjoint holes 
has been investigated intensively by various groups of researchers
(see e.g.~\cite{DevillersHKS2003,SakaiUrrutia2007,CanoGHSTU2015,BiniazMaheshwariSmid2017,HosonoUrabe2018}).
Two holes $X_1,X_2$ are called \emph{interior-disjoint} 
if their respective convex hulls are interior-disjoint.
Figure~\ref{fig:disjoint_holes} gives an illustration.
Interior-disjoint holes are also called \emph{compatible} in literature.
Note that a pair of interior-disjoint holes can share up to two vertices.

Interior-disjoint holes also play an important role in the study of other geometric objects
such as visibility graphs (see e.g.~\cite{DumitrescuTothPach2009}) or
flip graphs of triangulations on point sets (see e.g.~\cite{Pilz2018}).
Wagner and Welzl \cite{WagnerWelzl2019} quite recently
developed a framework for triangulations on planar point sets
which also allows the investigation of interior-disjoint holes.
Using their tools and results it is, for example, quite easy to derive that 
10 points always give a 4-hole and a 5-hole that are interior-disjoint.
In fact, when comparing the number of researchers working on the respective problems,
the interior-disjoint holes appear to be more of interest.

In a recent article, 
Hosono and Urabe~\cite{HosonoUrabe2018} 
summarized the current status and presented some new results.
They show that every set of 18 points contains two interior-disjoint 5-holes
and present a set of 13 points which does not contain interior-disjoint 5-holes.
By slightly adapting the SAT model from Section~\ref{sec:satmodel},
we managed to show that every set of 15 points contains two interior-disjoint 5-holes.
This bound is best possible because
the set of 14 points depicted in Figure~\ref{fig:n14} does not contain interior-disjoint 5-holes.
Also this set was found via our framework pyotlib (cf.\ Section~\ref{sec:results_three}).

\begin{figure}[htb]
\centering

\hfill
\begin{minipage}{0.55\textwidth}
\includegraphics[width=\textwidth]{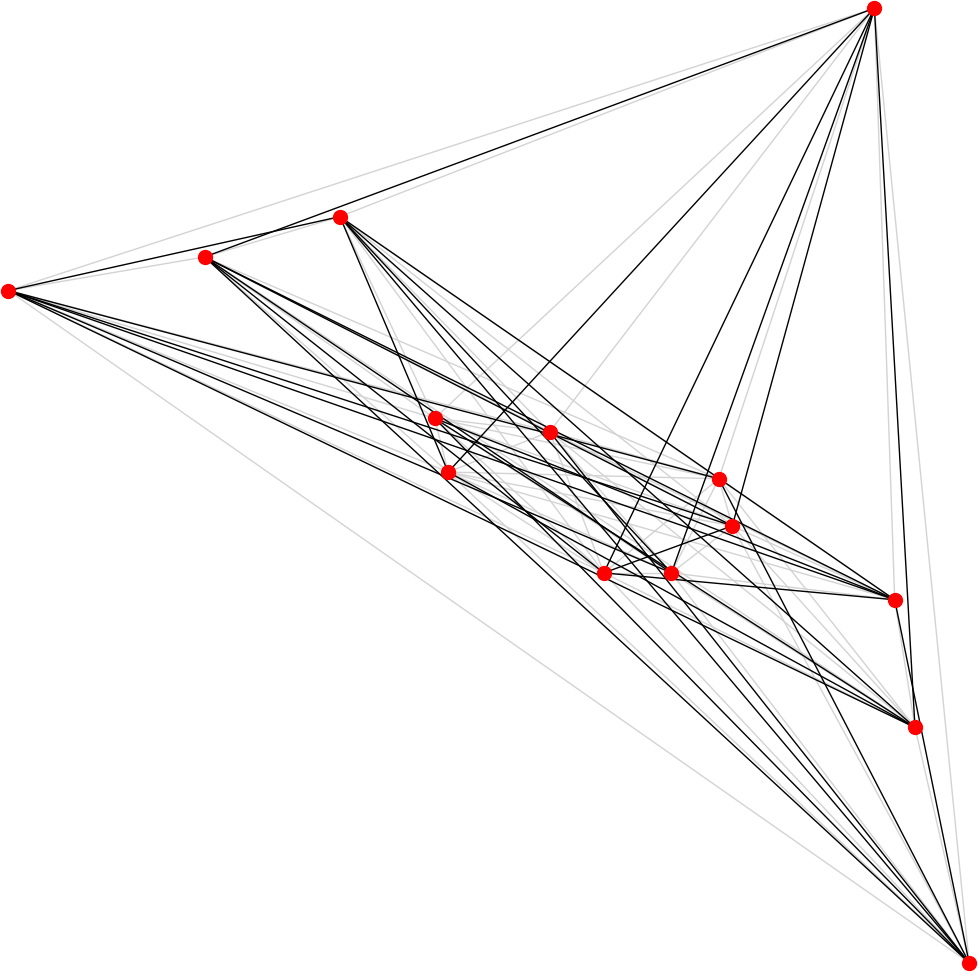}
\end{minipage}
\hspace{0.5cm}
\begin{minipage}{0.22\textwidth}
{
\small
\begin{verbatim}
142   0
  0 100
 29 105
 65  73
 63  81
 49 111
 88  58
 80  79
 98  58
107  65
105  72
134  35
131  54
128 142
\end{verbatim}
}
\end{minipage}
\hfill

\caption{A set of 14 points with no two interior-disjoint $5$-holes.
The coordinates are given on the right side.}
\label{fig:n14}
\end{figure}

Table~\ref{tab:2_interior_disjoin_holes} summarizes
the best possible bounds for two interior-disjoint holes \cite{HosonoUrabe2018}.
Again it is worth mentioning that all entries of the table (except for two interior-disjoint 5-holes)
can be verified using the order type database  \cite{AichholzerOTDB}
since at most 10 points are involved in the respective arguments.
We remark that, analogously to Section~\ref{sec:results_three}, 
one could further improve the bounds for three interior-disjoint holes.

\begin{table}[htb]
\def\arraystretch{1.2}
	\centering
	\begin{tabular}{r|lll}
			&3	&4	&5\\
	\hline
	3		&4	&5	&10\\
	4		&	&7	&10\\	
	5		&	&	&15*\\
	\end{tabular}
	\caption{Best possible bounds on the minimum number of points 
	such that every set of that many points contains two interior-disjoint holes of sizes $k_1$ and~$k_2$. 
	The entry marked with star (*) is new.}
	\label{tab:2_interior_disjoin_holes}
\end{table}

To be more specific on the changes of the SAT model for this variant: 
we slighly relaxed the contraints (8)
so that each of the two points $a$ and $b$, which span a separating line~$\ell$, 
can be contained in holes from both sides.
The program creating the SAT instance is also available 
as supplemental data and
on our website~\cite{website_disjoint_holes}.

\paragraph{Classical Erd\H{o}s--Szekeres:}
The computation time for the 
computer assisted proof by Szekeres and Peters~\cite{SzekeresPeters2006} for $g(6)=17$ 
was about 1500 CPU hours.
By slightly adapting the model from Section~\ref{sec:satmodel}
we have been able to confirm $g(6)=17$ 
using glucose and drat-trim
with about one hour of computation time on a 3GHz CPU.
To be more specific with the adaption of the model from Section~\ref{sec:satmodel}:
\begin{itemize}

 \item The constraints (6) are removed.
 
 \item The constraints (7) are adapted to ``(7*) 6-Gons''
 simply by testing 6-tuples instead of 5-tuples and 
 by dropping the requirement that ``triples form 3-holes''.
 
 \item The contraints (8) are removed.
\end{itemize}
Also this program is available as supplemental data and
on our website~\cite{website_disjoint_holes}.

An independent verification of $g(6)=17$ 
has been done by Mari\'c \cite{Maric2019},
who used a similar SAT framework.
While his program performs
slightly more efficient on verifying $g(6)=17$,
our model is more compact in size.
In fact, our program can be used to create an instance 
for testing $g(7)\stackrel{?}{=}33$ which only requires only 1.1~GB of disk space
(which was not possible in \cite{Maric2019}),
however, for determining whether $g(7)=33$
further ideas or more advanced SAT solvers seem to be required.
To be more precise about the differences in the two settings:
While Mari\'c assumed that points are indexed 
with respect to their convex hull peeling depth, 
we assumed that points sorted from left-to-right 
to break the symmetries in the search-space of the SAT solver.
Hence, in our model we could axiomize point sets via signotope axioms on 4-tuples
which saves a linear factor in the size of the instance 
compared to the axiomatization via 5-tuples such as used in \cite{Maric2019}
(cf.\ Chapters~2.1, 2.2, and~4.3 in   \cite{Scheucher2020Diss}).
Unfortunately, we do not see how 
Mari\'c's assumptions on the convex hull layers
can be combined with signotope axioms.

\paragraph{Counting 5-Holes:}

It is also possible to count occurences of certain substructures using SAT solvers.
For example to find point sets with as few 5-holes as possible,
we have introduced variables $X_{abcde;k}$ 
indicating whether the indices $1 \le a<b<c<d<e \le n$ 
form the $k$-th 5-hole in lexicographic order.
In particular, using SAT solvers we have been able to show that 
every set of 16 points contains at least 11 5-holes 
(cf.\ \cite{5holes_Socg2017,5holes_arXiv_version}).

\section*{Acknowledgements}

The author was supported by the DFG Grant FE~340/12-1 and 
by the internal research funding ``Post-Doc-Funding'' from Technische Universit\"at Berlin.
We thank Stefan Felsner, Linda Kleist, Felix Schr\"oder, Martin Balko, Adrian Dumitrescu, 
and Emo Welzl
for fruitful discussions and helpful comments.
Many thanks goes to Gyula K{\'a}rolyi for communicating his construction 
for $F_5(n) \leq \frac{n+1}{6}$, 
which in fact also shows $F_6(n) \leq \frac{n+1}{12}$ (cf.\ Section~\ref{sec:results_many}).
Last but not least, we would also like to thank 
the anonymous reviewers for their valuable comments
which further improved the quality of this article.

{
\small
\bibliographystyle{alphaabbrv-url}
\bibliography{bibliography}
}

\end{document}